\documentclass[12pt,reqno]{amsart}
\usepackage{amsmath}
\usepackage{amssymb}
\usepackage{amstext}
\usepackage{a4wide}
\usepackage{graphicx}
\usepackage{lineno}
\usepackage[numbers,sort&compress]{natbib}
\allowdisplaybreaks \numberwithin{equation}{section}
\usepackage{color}
\usepackage{cases}

\numberwithin{equation}{section}

\newtheorem{theorem}{Theorem}[section]

\newtheorem{corollary}[theorem]{Corollary}
\newtheorem{lemma}[theorem]{Lemma}
\newtheorem*{Yudovich's Theorem}{Yudovich's Theorem}

\theoremstyle{definition}

\theoremstyle{remark}
\newtheorem{remark}[theorem]{Remark}

\begin{document}

\title
[Asymptotic estimates for concentrated  vortex pairs]{Asymptotic estimates for concentrated vortex pairs}

 \author{Guodong Wang}

\address{Institute for Advanced Study in Mathematics, Harbin Institute of Technology, Harbin 150001, P.R. China}
\email{wangguodong@hit.edu.cn}


\begin{abstract}
In [Comm. Math. Phys. 324 (2013), 445--463], Burton-Lopes Filho-Nussenzveig Lopes studied the existence and stability of slowly traveling vortex pairs as maximizers of  the kinetic energy penalized by the impulse relative to a prescribed rearrangement class. In this paper, we prove that after a suitable scaling transformation the  maximization problem studied by Burton-Lopes Filho-Nussenzveig Lopes  in fact gives rise to a family of stable concentrated  vortex pairs approaching a pair of point vortices with equal magnitude and opposite signs.
The key ingredient of the proof is to deduce a uniform bound for the size of the supports of the scaled maximizers. 
 \end{abstract}

\maketitle

\section{Introduction and main results}
\subsection{Vorticity equation and vortex pair}
The motion of an incompressible inviscid fluid in a two-dimensional domain $D\subset\mathbb R^2$  can be described by the following Euler equations 
\begin{equation}\label{euler}
\begin{cases}
\partial_t\mathbf v+(\mathbf v\cdot\nabla)\mathbf v=-\nabla P, &\mathbf x=(x_1,x_2)\in D, \,t>0,\\
\nabla\cdot\mathbf v=0,
\end{cases}
\end{equation}
where $\mathbf v=(v_1,v_2)$ represents the velocity and $P$  the scalar pressure. Define the scalar vorticity $\omega=\partial_{x_1}v_2-\partial_{x_2} v_1$. Taking the curl on both sides of the  first equation of \eqref{euler} gives
\begin{equation}\label{vor}
\partial_t\omega+\mathbf v\cdot\nabla\omega=0.
\end{equation}
With suitable boundary condition, $\mathbf v$ can be recovered from $\omega$ via the Biot-Savart law. For example, when $D=\mathbb R^2$ and
$\mathbf v$ vanishes at infinity, the Biot-Savart law is given by
\begin{equation}\label{bsl1}
\mathbf v(t,\mathbf x)=-\frac{1}{2\pi}\int_{\mathbb R^2} \frac{(x-y)^\perp}{|x-y|^2}\omega(t,\mathbf y)d\mathbf y,
\end{equation}
where we used the symbol $\perp$ to denote clockwise rotation through $\pi/2$, i.e., $(z_1,z_2)^\perp=(z_2,-z_1)$ for any $(z_1,z_2)\in\mathbb R^2.$
When $D$ is the upper half-plane, i.e.,
\[D=\Pi:=\{\mathbf x=(x_1,x_2)\in\mathbb R^2\mid x_2>0\},\]
and $\mathbf v$ satisfies
\begin{equation}\label{bd1}
v_2=0 \mbox{ on } \partial\Pi=\{\mathbf x\in\mathbb R^2\mid x_2=0\},\quad |\mathbf v(\mathbf x)|\to 0 \mbox{ as } |\mathbf x|\to+\infty,
\end{equation}
 the Biot-Savart law is
\begin{equation}\label{bsl2}
\mathbf v=(\partial_{x_2}\mathcal G\omega,-\partial_{x_1}\mathcal G\omega),
\end{equation}
where  $\mathcal G$ is an integral operator given by
 \begin{equation}\label{green}
 \mathcal G\omega(\mathbf x)=\frac{1}{2\pi}\int_\Pi \ln\frac{|\mathbf x-\bar{\mathbf y}|}{|\mathbf x-\mathbf y|}\omega(\mathbf y)d\mathbf y,\quad  \bar{\mathbf y}=(y_1,-y_2).
 \end{equation}
Once the Biot-Savart law is determined, \eqref{vor} becomes a nonlinear transport equation with $\omega$ as the single unknown, which is usually called the \emph{vorticity equation}.
By Yudovich \cite{Y},  if the initial vorticity is compactly supported and bounded, then the vorticity equation \eqref{vor} has a unique global solution in the distributional  sense. 
See also Majda-Bertozzi \cite{MB} or Marchioro-Pulvirenti \cite{MP} for a modern proof.

In this paper, we are concerned with a special class of global solutions to the vorticity equation \eqref{vor},  called \emph{traveling vortex pairs}. A vortex pair  exhibits odd symmetry with respect to the $x_1$ axis and travels along the $x_1$ direction at a constant speed without changing its form. More mathematically, a vortex pair with speed $b$ is a solution $\omega$ to the vorticity equation \eqref{vor} with the  following form
\begin{equation}\label{omw}
\omega(t,\mathbf x)=\zeta(x_1-bt,x_2)-\zeta(x_1-bt,-x_2),
\end{equation}
where $\zeta$ is some bounded measurable function  supported in the upper half-plane $\Pi$
 (hence $\zeta=\omega\chi_\Pi,$ where  $\chi_\Pi$ denotes the characteristic function of  $\Pi$).

Given a solution $\omega$ to the vorticity equation \eqref{vor} in $\mathbb R^2$ with the Biot-Savart law \eqref{bsl1},  it is easy to check that  $\omega$ is oddly symmetric with respect to the $x_1$ axis if and only if  $\zeta=\omega\chi_\Pi$ satisfies the vorticity equation \eqref{vor} in $\Pi$ with the Biot-Savart law \eqref{bsl2}.
As a consequence, a traveling vortex pair in the full plane can  be viewed equivalently as a traveling solution to the vorticity equation \eqref{vor} in the upper half-plane with impermeability boundary condition on  $\partial\Pi$ and vanishing velocity at infinity.  For this reason, we shall mainly work in the upper half-plane in the rest of this paper.

By a formal computation, to seek a vortex pair of the form \eqref{omw},  it suffices to require $\zeta$ to satisfy
\begin{equation}\label{ffor}
\zeta=f(\mathcal G\zeta-bx_2)\quad\mbox{in }\Pi
\end{equation}
for some real function $f$.
 In terms of the stream function $\psi=\mathcal G\zeta$, \eqref{ffor} can be written equivalently as 
\begin{equation}\label{nep}
\begin{cases}
-\Delta \psi=f(\psi-bx_2),&\mathbf x\in\Pi,\\
\psi =0,&\mathbf x\in\partial \Pi,\\
|\nabla\psi|\to 0 \mbox{ as }|\mathbf x|\to+\infty.
\end{cases}
\end{equation}
In the field of nonlinear elliptic equations, the study for \eqref{nep} with prescribed $f$ has been an important research topic, and many existence results have been obtained via various methods. See \cite{Nor,SV,Yang1,Yang2} and the references listed therein.

A special case is when $f(s)=s^+$ in \eqref{nep}, where $s^+=\max\{s,0\},$ in which the solution to \eqref{ffor} can be explicitly expressed in terms of the Bessel functions of the first kind. See Burton's paper \cite{B50} for example. In this case, $\zeta$  is positive inside a semicircle and vanishes elsewhere, hence viewed in the full plane the vorticity $\omega$ given by \eqref{omw} is supported inside a circle. This solution was first  introduced by Lamb in \cite{Lamb} and is now usually called Lamb's circular vortex pair.

 Except for Lamb's circular vortex pair,
another important example is a pair of point vortices with equal magnitude and opposite signs traveling with a constant speed.
In this case, $\zeta$ has the form (up to a translation of the time variable)
 \[\zeta=\kappa\delta_{\mathbf z(t)},\quad \mathbf z(t)=(bt, d),\]
where $\kappa$ is a real number representing the vortex magnitude, $2d$ is the distance between the two  point vortices, and $\delta_\mathbf x $ denotes the unit Dirac measure  at $\mathbf x$.
 The vortex magnitude $\kappa$, the traveling speed $b$ and the distance $d$ necessarily satisfy
 \begin{equation}\label{rela1}
 4\pi bd=\kappa.
 \end{equation}
 See \cite{CLZ} for example.

It is worth mentioning that a pair of point vortices is just a formal singular solution of the vorticity equation \eqref{vor}, not in the distributional sense. Whether there exist a family of regular traveling solutions to the vorticity equation approximating a pair of point vortices is an interesting   problem. Related papers include  \cite{CLZ,HM,SV}. See also a recent interesting work \cite{HH} by Hassainia-Hmidi  on the existence of  concentrated asymmetric vortex pairs.

\subsection{Burton-Lopes Filho-Nussenzveig Lopes' results}
An effective way to construct traveling vortex pairs is to maximize the kinetic energy subject to some suitable constraints for the vorticity.  See \cite{B50,B6,B10,CLZ,SV,T} and the references therein. This paper is closely related to  Burton-Lopes Filho-Nussenzveig Lopes' work \cite{B6}, where existence and orbital stability  were proved by maximizing  the kinetic energy penalized by a small multiple impulse relative to all equimeasurable rearrangements of a given function. For the reader's convenience, we recall their results below.

Throughout this paper, denote by $L_{\rm b}^p(\Pi)$  the set of functions in $L^p(\Pi)$ with bounded support,   ${\rm supp}(\varrho)$   the essential support of $\varrho$ (see \cite{LL}, p.13 for the  definition of essential support of a measurable function), and $|\cdot|$   the two-dimensional Lebesgue measure.

Let $p\in(2,+\infty)$ be   fixed.  Consider a function $\varrho$  satisfying
 \begin{equation}\label{rho}
 \varrho \in L_{\rm b}^p(\Pi),\quad\varrho\geq 0\,\, \mbox { a.e. in }\Pi.
   \end{equation}
Denote by $\mathcal R(\varrho)$   the set of all equimeasurable  rearrangements of $\varrho$ in $\Pi,$ i.e.,
\[\mathcal R(\varrho)=\{v\in L^1(\Pi)\mid |\{\mathbf x\in \Pi\mid v(\mathbf x)>s\}| =|\{\mathbf x\in \Pi\mid \varrho(\mathbf x)>s\}|\,\,\forall\,s\in\mathbb R\},\]
and $\overline{\mathcal R(\varrho)^W}$ by the closure of $\mathcal R(\varrho)$ in the weak topology of $L^p(\Pi).$

Define the kinetic energy functional $E$ and the impulse functional $I$  by setting
\[E(v)=\frac{1}{2}\int_\Pi v(\mathbf x)\mathcal Gv(\mathbf x)d\mathbf x,\quad I(v)=\int_\Pi x_2v(\mathbf x)d\mathbf x.\]
In \cite{B6},  Burton-Lopes Filho-Nussenzveig Lopes studied the following maximization problem
\begin{equation}\label{mm1}
S_\lambda=\sup_{v\in \overline{\mathcal R(\varrho)^W}}(E-\lambda I)(v),
\end{equation}
where $\lambda $ is a positive  constant. In view of Lemma 1 in \cite{B11}, it is easy to check that $S_\lambda<+\infty$ for any $\lambda>0$.
Denote
\begin{equation}\label{sigmaa}
\Sigma_\lambda=\{v\in\overline{\mathcal R(\varrho)^W}\mid (E-\lambda I)(v)=S_\lambda\}.
\end{equation}
Burton-Lopes Filho-Nussenzveig Lopes  proved that
there exists some $\lambda_0,$ depending only on $\varrho$, such that for any $\lambda\in(0,\lambda_0),$
 $\varnothing\neq\Sigma_\lambda\subset \mathcal R(\varrho)$. Moreover, for any $\lambda\in(0,\lambda_0)$ the following properties for $\Sigma_\lambda$ hold.
 \begin{itemize}
\item [(a)] Every  $\zeta\in\Sigma_\lambda$ has bounded support and vanishes a.e. in $\{\mathbf x\in\Pi\mid \mathcal G\zeta(\mathbf x)-\lambda x_2\leq 0\}.$

\item[(b)] Every $\zeta\in\Sigma_\lambda$ is a translation of some function that is Steiner-symmetric about the $x_2$ axis (i.e., even in $x_1$ and decreasing in
$|x_1|$).
\item [(c)] For every  $\zeta\in\Sigma_\lambda$, there exist some increasing function $f:\mathbb R\to\mathbb R\cup\{\pm\infty\}$, depending on $\zeta$,  such that $\zeta=f(\mathcal G\zeta-\lambda x_2)$  a.e. in $\Pi$.
\item [(d)] $\Sigma_\lambda$ is orbitally stable in the following sense: for any  $A>|{\rm supp}(\varrho)|$ and  any $\epsilon>0,$ there exists some $\delta>0,$ depending on $\varrho, \lambda, \epsilon$ and $A,$ such that for any $\omega_0\in L_{\rm b}^\infty(\Pi)$ with $|{\rm supp}(\omega_0)|< A$, it holds that
\[\inf_{v\in \Sigma_\lambda}|I(\omega_0)-I(v)|+\inf_{v\in\Sigma_\lambda}\|\omega_0-v\|_{L^2(\Pi)}<\delta\Longrightarrow \inf_{v\in\Sigma_\lambda}\|\omega_t-v\|_{L^2(\Pi)}<\epsilon,\,\,\forall \,t\ge 0,\]
whenever $\omega_t$ is any $L^p$-regular solution to the vorticity equation \eqref{vor}\eqref{bsl2} in $\Pi$ with initial vorticity $\omega_0.$
\end{itemize}

Here by an \emph{$L^p$-regular solution}, we mean a distributional solution $\omega\in L^\infty_{\rm loc}((0,+\infty), L^1\cap L^p(\Pi))$ to the vorticity equation \eqref{vor}\eqref{bsl2} such that   $E(\omega(t,\cdot))$ and  $I(\omega(t,\cdot))$ are constants. 
See \cite{B6} or \cite{B10} for the precise definition. 

\begin{remark}
In view of the property (c), it can be verified that for any $\zeta\in \Sigma_\lambda$,   $\zeta(x_1-\lambda t,x_2)$ solves the vorticity equation in the distributional sense (see \cite{B11}, Section 5 for a detailed proof),  thus yielding a traveling vortex pair with speed $\lambda.$
 
\end{remark}

\begin{remark}
Here we take $\Sigma_\lambda$ as a family of traveling solutions with the fluid velocity vanishing at infinity. This is a little different from the viewpoint in \cite{B6},
where $\Sigma_\lambda$ was viewed as a family of steady vortex pairs in an irrotational background flow approaching a uniform stream at infinity.  These two different viewpoints are equivalent  in mathematics.
\end{remark}

\subsection{Scaled problem and main results}

Although  Burton-Lopes Filho-Nussenzveig Lopes  obtained the existence and stability for slowly traveling vortex pairs, they did not  study the limiting behavior of the set $\Sigma_\lambda$ as the traveling speed $\lambda $  vanishes. We will see this problem will lead to (after a suitable scaling transformation) a family of concentrated traveling vortex pairs approximating a pair of point vortices, which  is of particular interest physically.

Let $\varrho$ satisfy \eqref{rho}. Denote 
\begin{equation}\label{sc8}
\kappa=\int_\Pi\varrho(\mathbf x)d\mathbf x.
\end{equation}
 For any $\varepsilon>0$, define
\begin{equation}\label{sc9}
\varrho^\varepsilon(\mathbf x)=\frac{1}{\varepsilon^2}\varrho\left(\frac{\mathbf x}{\varepsilon}\right),\quad \mathbf x\in\Pi.
\end{equation}
Then   for any $s\in[1,+\infty]$, it holds that 
\begin{equation}\label{sc10}
\|\varrho^\varepsilon\|_{L^s(\Pi)}=\varepsilon^{2/s-2}\|\varrho\|_{L^s(\Pi)}.
\end{equation}
Let $q>0$ be  fixed. Consider the maximization problem
\begin{equation}\label{mm2}
T_{\varepsilon}=\sup_{v\in\overline{\mathcal R(\varrho^\varepsilon)^W}}(E-qI)(v).
\end{equation}
Define
\begin{equation}\label{xi0ff}
\Xi_\varepsilon=\{v\in\overline{\mathcal R(\varrho^\varepsilon)^W}\mid (E-qI)(v)=T_{\varepsilon}\},
\end{equation}
\begin{equation}\label{xi0}
\Xi^0_\varepsilon=\{\zeta\in\Xi_\varepsilon\mid \zeta(x_1,x_2)=\zeta(-x_1,x_2),\,\,\forall\,\mathbf x=(x_1,x_2)\in\Pi\}.
\end{equation}

Our main result is as follows.
\begin{theorem}\label{thm1}
Let $\varrho$ satisfy \eqref{rho} with $p\in(2,+\infty)$, $q>0$ be given, and $\Xi_\varepsilon,\Xi^0_\varepsilon$ be defined as above.
Then there exists some $\varepsilon_0,$ depending only on $\varrho$ and $q$, such that for any $\varepsilon\in(0,\varepsilon_0)$ the following results hold true.
\begin{itemize}
\item [(i)] $\varnothing\neq \Xi_\varepsilon\subset \mathcal R(\varrho^\varepsilon)$. Moreover, for any    $\zeta\in \Xi_\varepsilon$, $\zeta(x_1-q t,x_2)$ solves the vorticity equation \eqref{vor}\eqref{bsl2} in the distributional sense.
\item[(ii)] Each $\zeta\in\Xi_\varepsilon$ is a translation of some  function that is Steiner-symmetric about the $x_2$ axis, has bounded support, vanishes a.e. in $\{\mathbf x\in\Pi\mid \mathcal G\zeta(\mathbf x)-qx_2\leq 0\},$ and satisfies $\zeta=f(\mathcal G\zeta-qx_2)$  a.e. in  $\Pi$ for some increasing function $f:\mathbb R\to\mathbb R\cup\{\pm\infty\}$.
\item[(iii)]  $\Xi_\varepsilon$ is orbitally stable in the following sense: for any  $A>\varepsilon^2|{\rm supp}(\varrho)|$ and   any $\epsilon>0,$ there exists some $\delta>0,$ depending on $\varrho, \varepsilon, \epsilon$ and $A,$ such that for any $\omega_0\in L_{\rm b}^\infty(\Pi)$ with $|{\rm supp}(\omega_0)|< A$, it holds that
\[\inf_{v\in \Xi_\varepsilon}|I(\omega_0)-I(v)|+\inf_{v\in\Xi_\varepsilon}\|\omega_0-v\|_{L^2(\Pi)}<\delta\Longrightarrow \inf_{v\in\Xi_\varepsilon}\|\omega_t-v\|_{L^2(\Pi)}<\epsilon,\,\,\forall \,t\ge 0,\]
whenever $\omega_t$ an $L^p$-regular solution to the vorticity equation \eqref{vor}\eqref{bsl2} in $\Pi$ with initial vorticity $\omega_0.$
\item[(iv)] There exists a positive number $C,$ depending only on $\varrho$ and $q$, such that
\begin{equation}\label{dses}
{\rm diam(supp}(\zeta))\leq C\varepsilon,\,\,\forall\, \zeta\in\Xi_\varepsilon.
\end{equation}
Here ${\rm diam}(\cdot)$ denotes the diameter of some set.
\item[(v)] For any $\zeta\in\Xi^0_\varepsilon$, denote  
\[ \mathbf x^{\zeta,\varepsilon}=\frac{1}{ \kappa}\int_\Pi \mathbf x\zeta(\mathbf x)d\mathbf x.\]
Then as $\varepsilon\to0^+,$ 
 \[\mathbf x^{\zeta,\varepsilon}\to \hat{\mathbf x}=\left(0, \frac{\kappa}{4\pi q}\right), \] 
  uniformly with respect to the choice of  $\zeta\in\Xi^0_{\varepsilon}$. More precisely, for any $\epsilon>0,$ there exists some  $\varepsilon_1\in(0,\varepsilon_0),$ such that for any $\varepsilon\in(0,\varepsilon_1),$ it holds that
\[|\mathbf x^{\zeta,\varepsilon}-\hat{\mathbf x}|<\epsilon,\quad\forall\,\zeta\in\Xi^0_\varepsilon.\]

\item[(vi)] For $\zeta\in\Xi^0_\varepsilon$, extend $\zeta$ to $\mathbb R^2$ such that $\zeta\equiv0$ in the lower half-plane and define 
\[\nu^{\zeta,\varepsilon}(\mathbf x)=\varepsilon^2\zeta(\varepsilon\mathbf x+\mathbf x^{\zeta,\varepsilon}).\] Denote by $\varrho^*$  the symmetric-decreasing rearrangement of $\varrho$ with respect to the origin. Then   $\nu^ {\zeta,\varepsilon}\to\varrho^*$ in $L^p(\mathbb R^2)$ as $\varepsilon\to0^+,$ uniformly with respect to the choice of $\zeta\in\Xi^0_{\varepsilon}$.
More precisely, for any $\epsilon>0,$ there exists some  $\varepsilon_2\in(0,\varepsilon_0),$ such that for any $\varepsilon\in(0,\varepsilon_2),$ it holds that
\[\|\nu^{\zeta,\varepsilon}-\varrho^*\|_{L^p(\mathbb R^2)}<\epsilon,\quad\forall\,\zeta\in\Xi^0_\varepsilon.\]

\end{itemize}
\end{theorem}

By (iv) and (v) in Theorem \ref{thm1}, we see that as $\varepsilon\to0^+$ ${\rm supp}(\zeta)$  ``shrinks" to $\hat{\mathbf x}$  uniformly with respect to the choice of  $\zeta\in\Xi^0_{\varepsilon}$.    More precisely, for any $r>0$, there exists some $\varepsilon_3\in(0,\varepsilon_0),$ such that for any $\varepsilon\in(0,\varepsilon_3),$ it holds that
\[{\rm supp}(\zeta)\subset B_r(\hat{\mathbf x}),\quad\forall\,\zeta\in\Xi^0_\varepsilon.\]

In Theorem \ref{thm1}, the assertions (i)-(iii)  are easy consequences of Burton-Lopes Filho-Nussenzveig Lopes' results. Our main purpose in this paper is to prove the asymptotic estimates (iv)-(vi). 
Our method of proof is mostly inspired by  Turkington's paper   \cite{T}, where concentrated Euler flows with piecewise  constant vorticity  were constructed by solving a similar maximization problem.
The key point in Turkington's method was to obtain a suitable lower bound for the Lagrangian multiplier, which was  achieved by deducing a suitable lower bound for the  energy in the whole domain and a uniform upper bound for the energy on the vortex core. See also \cite{CLZ,CWZ,CW1,CW2,EM0, EM} for some further applications or developments of this method.
 In this paper, following Turkington's idea, it is not hard to get the estimate for the  Lagrangian multiplier $\mu_\zeta$  (defined by \eqref{lmp} in Section 3). However, even with the desired estimate for $\mu_\zeta$, we are still not able to prove (iv)-(vi)
 due to the unboundedness of the upper half-plane.
 To overcome this difficulty, we need to accomplish the most important step, i.e., to derive a uniform bound for the diameter of the supports of the maximizers (see Lemma \ref{lem307}),
which  is also  the hardest part of the proof. This  step is achieved by using an adaption of Turkington's method  and improving  some    estimates first proved by Burton  in \cite{B6,B11}.

 An alternative way to construct concentrated vortex pairs is to maximize $E-qI$ over
 \[\{v\in\mathcal R(\varrho^\varepsilon)\mid {\rm supp}(v)\subset B_{r_0}(\hat{ \mathbf x}),\,\,v\mbox{ is Steiner-symmetric about the $x_2$ axis}\},\]
 where $r_0$ is a fixed  positive number. 
In this case, the difficulty caused by the unboundedness of the upper half-plane can be avoided. However, as the cost, proving stability   becomes difficult. We will discuss this issue in detail  in Section 4.

As a consequence of Theorem \ref{thm1}, we can easily prove the following asymptotic behavior for the vortex pairs obtained by Burton-Lopes Filho-Nussenzveig Lopes as the traveling speed vanishes.
For convenience, we denote 
\begin{equation}\label{sigma0}
\Sigma^0_\lambda=\{\zeta\in\Sigma_\lambda\mid \zeta(x_1,x_2)=\zeta(-x_1,x_2),\,\,\forall\,\mathbf x=(x_1,x_2)\in\Pi\}.
\end{equation}

\begin{corollary}\label{coro1}
Let $\varrho$ satisfy \eqref{rho} with $p\in(2,+\infty)$,   and $\Sigma_\lambda,\Sigma_\lambda^0$ be defined by \eqref{sigmaa}, \eqref{sigma0}.
Then the following properties  hold true.
\begin{itemize}
\item[(i)] There exists some  $C>0$, depending only on $\varrho$, such that
\[{\rm diam(supp}(\zeta)) \leq C,\quad\forall\,\zeta\in\Sigma_\lambda.\]
\item[(ii)] Denote   
\[\mathbf x^{\zeta,\lambda}= \frac{1}{\kappa}\int_\Pi \mathbf x\zeta(\mathbf x)d\mathbf x.\]
Then for $\zeta\in\Sigma^0_\lambda$, 
$\mathbf x^{\zeta,\lambda}=(x_1^{\zeta,\lambda}, x_2^{\zeta,\lambda})$ satisfies 
\[x_1^{\zeta,\lambda}\equiv0,\quad\lim_{\lambda\to0^+}\lambda x_2^{\zeta,\lambda}=\frac{ \kappa }{4\pi},\]  
where the convergence is uniform  with respect to the choice of  $\zeta\in\Sigma^0_{\lambda}$, i.e., for any $\epsilon>0,$ there exists some  $\lambda_1\in(0,\lambda_0),$ such that for any $\lambda\in(0,\lambda_1),$ it holds that
\[\left|\lambda x_2^{\zeta,\lambda}-\frac{\kappa}{4\pi}\right|<\epsilon,\quad\forall\,\zeta\in\Sigma^0_\lambda.\]

\item[(iii)] For $\zeta\in\Sigma^0_\lambda$, extend $\zeta$ to $\mathbb R^2$ such that $\zeta=0$ in the lower half-plane and define  $\nu^{\zeta,\lambda}(\mathbf x)=\zeta(\mathbf x+\mathbf x^{\zeta,\lambda})$. Then  
$\nu^{\zeta,\lambda}\to\varrho^*$ in $L^p(\mathbb R^2)$ as $\lambda\to0^+,$ uniformly with respect to the choice of $\zeta\in\Sigma_\lambda^0.$ More precisely, for any $\epsilon>0,$ there exists some  $\lambda_2\in(0,\lambda_0),$ such that for any $\lambda\in(0,\lambda_2),$ it holds that
\[\|\nu^{\zeta,\lambda}-\varrho^*\|_{L^p(\mathbb R^2)}<\epsilon,\quad\forall\,\zeta\in\Sigma^0_\lambda.\]

\end{itemize}
\end{corollary}

This paper is organized as follows. In Section 2, we give some preliminaries for later use. In Section 3, we give the proofs of Theorem \ref{thm1} and Corollary \ref{coro1}. In Section 4, we present some further discussions.

\section{Preliminaries}
In this section, we give some preliminaries results for later use. 
\begin{lemma}[\cite{B6}, Lemma 10]\label{lemm1}
Let  $w \in L_{\rm b}^p(\Pi)$ be nonnegative and $\theta$ be a positive number. Denote $\mathcal M_\theta $  the set of maximizers of $E-\theta I$ over $\overline{\mathcal R(w)^W}$. Then there exists some $\theta_0>0,$ depending only on $w,$ such that for any $\theta\in(0,\theta_0),$ it holds that $\varnothing\neq \mathcal M_\theta\subset \mathcal R(w).$
\end{lemma}

\begin{lemma}[\cite{B6}, Theorem 2]\label{lemm2}
Let  $w\in L_{\rm b}^p(\Pi)$ be nonnegative and $\theta$ be a positive number. Denote $\mathcal M_\theta $  the set of maximizers of $E-\theta I$ over $\overline{\mathcal R(w)^W}$.  If
$  \mathcal M_\theta\subset \mathcal R(w),$ then $\mathcal M_\theta\neq\varnothing,$
and each  $\zeta\in \mathcal M_\theta$ is a translation of some function that is Steiner-symmetric about the $x_2$ axis, has bounded support, vanishes a.e. in $\{\mathbf x\in\Pi\mid \mathcal G\zeta-\theta x_2\leq 0\},$ and satisfies $\zeta=f(\mathcal G\zeta-\theta x_2)$  a.e. in $\Pi$ for some increasing function $f:\mathbb R\to\mathbb R\cup\{\pm\infty\}$.

\end{lemma}

\begin{lemma}[\cite{B6}, Theorem 1]\label{lemm3}
Let  $w\in L_{\rm b}^p(\Pi)$ be nonnegative and $\theta$ be a positive number. Denote $\mathcal M_\theta $  the set of maximizers of $E-\theta I$ over $\overline{\mathcal R(w)^W}$.  If
$\varnothing\neq \mathcal M_\theta\subset \mathcal R(w),$ then $\mathcal M_\theta$ is orbitally stable in the following sense:   for any fixed positive number $A>|{\rm supp}(w)|$ and  any $\epsilon>0,$ there exists some $\delta>0,$ depending only on $w,\theta,\epsilon$ and $A$, such that for any $\omega_0\in L_{\rm b}^p(\Pi)$ with $|{\rm supp}(\omega_0)|< A$, it holds that
\[\inf_{v\in \mathcal M_\theta}|I(\omega_0)-I(v)|+\inf_{v\in \mathcal M_\theta}\|\omega_0-v\|_{L^2(\Pi)}<\delta\Longrightarrow \inf_{v\in \mathcal M_\theta}\|\omega_t-v\|_{L^2(\Pi)}<\epsilon,\,\,\forall \,t\ge 0,\]
whenever $\omega_t$ is an $L^p$-regular solution to the vorticity equation \eqref{vor}\eqref{bsl2} in $\Pi$ with initial vorticity $\omega_0.$

\end{lemma}

\begin{lemma}[\cite{B10}, Lemma 5]\label{lemm4}
Let $2<s<+\infty.$ Then there exists a positive number $K$, depending only on $s$, such that for any non-negative $v\in L^1(\Pi)\cap L^s(\Pi) $ that is Steiner-symmetric about the $x_2$ axis, it holds that
\[\mathcal Gv(\mathbf x)\leq K\left(I(v)+\|v\|_{L^1(\Pi)}+\|v\|_{L^s(\Pi)}\right)x_2|x_1|^{-1/(2s)}, \quad\forall \,\mathbf x=(x_1,x_2)\in\Pi,\,|x_1|\geq 1.\]
\end{lemma}

The following  Sobolev inequality will be used in Lemma \ref{lem302}. 
\begin{lemma}\label{soboemb}
There exists a generic positive constant $S$ such that 
\begin{equation}\label{soboemb1}
\|u\|_{L^2(\mathbb R^2)}\leq S\|\nabla \phi\|_{L^1(\mathbb R^2)},\quad\forall \,u\in W^{1,1}(\mathbb R^2).
\end{equation}
\end{lemma}
The proof of Lemma \ref{soboemb} can be found in many textbooks. See \cite{EV},  p. 162 for example.

For any Lebesgue measurable function $u$, we use $u^*$ to denote its symmetric-decreasing rearrangement with respect to the origin. See \cite{LL}, \S 3.3 for the precise definition.   We will need the following two rearrangement inequalities.
\begin{lemma}[\cite{LL}, \S 3.4]\label{rri1}
Let $u,v$ be nonnegative Lebesgue measurable functions on $\mathbb R^2$. Then
\[\int_{\mathbb R^2}uvdx\leq \int_{\mathbb R^2}u^*v^*dx.\]
\end{lemma}

\begin{lemma}[\cite{LL}, \S 3.7]\label{rri2}
Let $u,v,w$ be nonnegative Lebesgue measurable functions on $\mathbb R^2$. Then
\[\int_{\mathbb R^2}\int_{\mathbb R^2}u(x)v(x-y)w(y)dxdy\leq \int_{\mathbb R^2}\int_{\mathbb R^2}u^*(x)v^*(x-y)w^*(y)dxdy.\]
\end{lemma}

The following lemma is a  direct consequence of Lemma 3.2 in \cite{BGu}.
\begin{lemma}[\cite{BGu}, Lemma 3.2]\label{bgu}
Let $\{u_n\}_{n=1}^{+\infty}\subset L^2(\mathbb R^2)$ such that for each $n$
\begin{equation}\label{lily}
 u_n(\mathbf x)\geq 0  \mbox{ a.e. }\mathbf x\in\mathbb R^2,\quad\int_{\mathbb R^2}\mathbf x u_n(\mathbf x)d\mathbf x=\mathbf 0,\quad {\rm supp}(u_n)\subset B_{\alpha}(\mathbf 0) 
 \end{equation}
for some $\alpha>0$.
If $u_n\rightharpoonup u$ and $u^*_n\rightharpoonup v$ for some $u,v\in L^2(\mathbb R^2),$ then 
\[\int_{\mathbb R^2}\int_{\mathbb R^2}\ln\frac{1}{|\mathbf x-\mathbf y|}u(\mathbf x)u(\mathbf y)d\mathbf xd \mathbf y\leq \int_{\mathbb R^2}\int_{\mathbb R^2}\ln\frac{1}{|\mathbf x-\mathbf y|}v(\mathbf x)v(\mathbf y)d\mathbf xd \mathbf y.\]
Moreover, the equality holds if and only $u=v$.
\end{lemma}

\section{Proofs}

In this section, we give the proofs of Theorem \ref{thm1} and Corollary \ref{coro1}.

First we show by a simple calculation that maximizing $E-qI$ over $\mathcal R(\varrho^\varepsilon)$ is equivalent to  maximizing $E-{\varepsilon}qI$ over $\mathcal R(\varrho)$.
\begin{lemma}\label{lem31}
For any $\varepsilon>0,$ it holds that
\[S_{\varepsilon q}=T_\varepsilon,\quad\Sigma_\varepsilon=\{ v\in\mathcal R(\varrho)\mid v(\mathbf x)=\varepsilon^2\zeta(\varepsilon\mathbf x)\mbox{ for some }\zeta\in \Xi_\varepsilon\}.\]
\end{lemma}

\begin{proof}
For any $v\in \mathcal R(\varrho^\varepsilon),$ define $w(x)={\varepsilon^2}v({\varepsilon}\mathbf x).$ It is clear that $w\in\mathcal R(\varrho).$
By a direct calculation, we have
\begin{align*}
(E-qI)(v)&=\frac{1}{4\pi}\int_\Pi\int_\Pi\ln\frac{|\mathbf x-\bar{\mathbf y}|}{|\mathbf x-\mathbf y|}v(\mathbf x)v(\mathbf y)d\mathbf xd\mathbf y-q\int_\Pi x_2v(\mathbf x)d\mathbf x\\
&=\frac{1}{4\pi}\int_\Pi\int_\Pi\ln\frac{|\mathbf x-\bar{\mathbf y}|}{|\mathbf x-\mathbf y|}w(\mathbf x)w(\mathbf y)d\mathbf xd\mathbf y-\varepsilon q\int_\Pi x_2w(\mathbf x)d\mathbf x\\
&=(E-\varepsilon qI)(w).
\end{align*}
Hence the desired result follows.
\end{proof}

As a consequence of Lemma \ref{lem31}, $\Xi_\varepsilon$ is not empty for any $\varepsilon\in(0,\varepsilon_0)$.
Below   we assume that  $\varepsilon_0$ is sufficiently small as needed and $\varepsilon\in(0,\varepsilon_0)$.

By Lemma \ref{lemm2}, for any $\zeta\in\Xi_\varepsilon,$ there exists some increasing function $f_\zeta$, depending on $\zeta,$ such that $\zeta=f_\zeta(\mathcal G\zeta-qx_2)$  a.e. in $\Pi$.
 Define the Lagrangian multiplier
 \begin{equation}\label{lmp}
 \mu_\zeta=\inf\{s\in\mathbb R\mid f_\zeta(s)>0\}.
 \end{equation}
 Since $\zeta$ vanishes a.e. in $\{\mathbf x\in\Pi\mid \mathcal G\zeta-qx_2\leq 0\},$ we deduce  that $f_\zeta\equiv 0$ on $(-\infty,0]$, which implies $\mu_\zeta\geq 0.$
 
 For any $\zeta\in\Xi_\varepsilon,$ denote by $V_\zeta$ the vortex core related to $\zeta,$ that is, 
\begin{equation}
V_\zeta=\{\mathbf x\in\Pi \mid \zeta(\mathbf x)>0\}.
\end{equation} 

 \begin{lemma}\label{ccc2}
It holds that
 \begin{equation}\label{lesss}
V_\zeta =\{\mathbf x\in\Pi\mid \mathcal G\zeta(\mathbf x)-qx_2>\mu_\zeta\},\quad \forall\,\zeta\in\Xi_\varepsilon.
\end{equation}
 \end{lemma}
\begin{proof}
Fix $\zeta\in\Xi_\varepsilon$.  By the definition of $\mu_\zeta$, we have 
 \[\zeta>0 \mbox{ a.e. in } \{\mathbf x\in\Pi\mid \mathcal G\zeta(\mathbf x)-qx_2>\mu_\zeta\}\]
 and
 \[\zeta=0 \mbox{ a.e. in } \{\mathbf x\in\Pi\mid \mathcal G\zeta(\mathbf x)-qx_2<\mu_\zeta\}.\]
On the set $\{\mathbf x\in\Pi\mid \mathcal G\zeta(\mathbf x)-qx_2=\mu_\zeta\}$, taking into account  the property that all the derivatives of a Sobolev function  vanishes  on its level  set (see \cite{EV}, p. 153 for example), we get
 \[\zeta=-\Delta(\mathcal G\zeta-qx_2)=0 \quad\mbox{ a.e. in } \{\mathbf x\in\Pi\mid \mathcal G\zeta(\mathbf x)-qx_2=\mu_\zeta\}.\]
 Hence \eqref{lesss} is proved. 
\end{proof}

 Below we give some asymptotic estimates as $\varepsilon\to0^+.$   As in Section 1,  denote $\kappa=\|\varrho\|_1.$

\begin{lemma}\label{lem301}
There exists some $C_1>0,$ depending only on $\varrho$ and $q$, such that
\[(E-qI)(\zeta)\geq -\frac{\kappa^2}{4\pi}\ln\varepsilon-C_1,\quad\forall \,\zeta\in\Xi_\varepsilon.\]
\end{lemma}
\begin{proof}
 Let $\varrho^*$ be the radially decreasing rearrangement of $\varrho$ with respect to the origin.  Then by choosing $a>0$ such that $\pi a^2=|\{\mathbf x\in\Pi\mid \varrho(\mathbf x)>0\}|,$
we have  
\[{\rm supp}(\varrho^*)=\overline{B_a(\mathbf 0)}.\] 
  Define 
  \[v(\mathbf x)=\frac{1}{ \varepsilon^2}\varrho^*\left(\frac{\mathbf x-\hat{\mathbf x}}{\varepsilon}\right), \quad\hat{\mathbf x}=(0,2a).\]
 Then obviously ${\rm supp}(v)= \overline{B_{a\varepsilon}(\hat{\mathbf x})}$ and  $v\in\mathcal R(\varrho^\varepsilon)$. Hence
\[(E-qI)(\zeta)\geq (E-qI)(v),\quad\forall\,\zeta\in\Xi_\varepsilon.\]
By a direct calculation, we have
\begin{align*}
(E-qI)(v)=&\frac{1}{4\pi}\int_{\Pi}\int_{\Pi}\ln\frac{|\mathbf x-\bar{\mathbf y}|}{|\mathbf x-\mathbf y|}v(\mathbf x)v(\mathbf y)d\mathbf xd\mathbf y-q\int_{\Pi}x_2v(\mathbf x) d\mathbf x\\
=&\frac{1}{4\pi}\int_{B_{\varepsilon a}(\hat{\mathbf x})}\int_{B_{\varepsilon a}(\hat{\mathbf x})}\ln\frac{|\mathbf x-\bar{\mathbf y}|}{|\mathbf x-\mathbf y|}v(\mathbf x)v(\mathbf y)d\mathbf xd\mathbf y-q\int_{B_{\varepsilon a}(\hat{\mathbf x})}x_2v (\mathbf x) d\mathbf x
\end{align*}
For the first integral, since $|\mathbf x-\mathbf y|\leq 2\varepsilon a$ and $|\mathbf x-\bar{\mathbf y}|\geq 4a-2\varepsilon a\geq 2a$ for any $\mathbf x,\mathbf y\in B_{\varepsilon a}(\hat{\mathbf x}),$ we have
\begin{align*}
\frac{1}{4\pi}\int_{B_{\varepsilon a}(\hat{\mathbf x})}\int_{B_{\varepsilon a}(\hat{\mathbf x})}\ln\frac{|\mathbf x-\bar{\mathbf y}|}{|\mathbf x-\mathbf y|}v(\mathbf x)v(\mathbf y)d\mathbf xd\mathbf y
&\geq \frac{1}{4\pi}\int_{B_{\varepsilon a}(\hat{\mathbf x})}\int_{B_{\varepsilon a}(\hat{\mathbf x})}\ln\left(\frac{2a}{2\varepsilon a}\right)v(\mathbf x)v(\mathbf y)d\mathbf xd\mathbf y\\
&=-\frac{\kappa^2}{4\pi}\ln\varepsilon.
\end{align*}
For the second integral, by symmetry we have
\[\int_{B_{\varepsilon a}(\hat{\mathbf x})}x_2v (\mathbf x) d\mathbf x
=2a\kappa q.\]
Hence the proof is finished.
\end{proof}

\begin{lemma}\label{lem302}
There exists some $C_2>0,$ depending only on $\varrho$ and $\varrho$, such that
\[\int_\Pi\zeta(\mathcal G\zeta-qx_2-\mu_\zeta)d\mathbf x\leq C_2,\quad\forall \,\zeta\in\Xi_\varepsilon.\]
\end{lemma}
\begin{proof}
Fix $\zeta\in\Xi_\varepsilon$.
For simplicity, denote $\phi=\mathcal G\zeta-qx_2-\mu_\zeta$ and $\phi^+=\max\{\phi,0\}.$ Then by \eqref{lesss} we have
\begin{equation}
\{\mathbf x\in\Pi\mid \phi(\mathbf x)>0\}=V_\zeta.
\end{equation}
Since $\mu_\zeta\geq 0$ and $\mathcal G\zeta-q x_2\equiv 0$ on $\partial \Pi,$ we see that  $\phi^+$ vanishes on $\partial\Pi$.
Therefore we can apply  integration by parts to get
\begin{equation}\label{ooh1}
\int_\Pi\phi\zeta d\mathbf x=\int_\Pi\zeta\phi^+ d\mathbf x=\int_\Pi|\nabla \phi^+|^2 d\mathbf x.
\end{equation}
On the other hand,  
\begin{align}
\int_\Pi\phi\zeta d\mathbf x& \leq \|\zeta\|_{p}\|\phi^+\|_{L^q(\Pi)}  \quad(q:=p/(p-1)) \label{uuio}\\
&= \|\varrho^\varepsilon\|_{L^p(\Pi)}\|\phi^+\|_{L^q(\Pi)}\\
&= \varepsilon^{-2/q}\|\varrho\|_{L^p(\Pi)} \|\phi^+\|_{L^q(\Pi)}\label{ii33}\\
&\leq \varepsilon^{-2/q}\|\varrho\|_{L^p(\Pi)} |V_\zeta|^{\frac{2-q}{2q}}\|\phi^+\|_2\label{ii44}\\
&\leq S\varepsilon^{-2/q}\|\varrho\|_{L^p(\Pi)} |V_\zeta|^{\frac{2-q}{2q}}\|\nabla\phi^+\|_{L^1(\Pi)}\label{ii55}\\
&\leq S\varepsilon^{-2/q}\|\varrho\|_{L^p(\Pi)} |V_\zeta|^{\frac{2-q}{2q}}|V_\zeta|^{\frac{1}{2}}\|\nabla\phi^+\|_{L^2(\Pi)}\label{ii66}\\
&=S |\{\mathbf x\in\Pi\mid \varrho(\mathbf x)>0\}|^{\frac{1}{q}} \|\varrho\|_{L^p(\Pi)}\|\nabla\phi^+\|_{L^2(\Pi)},\label{ii88}
\end{align}
where  $S$ is the   positive constant in Lemma \ref{soboemb}.
Note that we have used H\"older's inequality in  \eqref{uuio}, \eqref{ii44}, \eqref{ii66}, the Sobolev inequality  \eqref{soboemb1}
in \eqref{ii55}, and the following fact in \eqref{ii88}
\[|V_\zeta|=|\{\mathbf x\in\Pi\mid \zeta(\mathbf x)>0\}|=|\{\mathbf x\in\Pi\mid \varrho^\varepsilon(\mathbf x)>0\}|=\varepsilon^2|\{\mathbf x\in\Pi\mid \varrho(\mathbf x)>0\}|.\]
Combining \eqref{ooh1} and \eqref{ii88}, we obtain
\[\int_\Pi\phi\zeta d\mathbf x\leq S^2 |\{\mathbf x\in\Pi\mid \varrho(\mathbf x)>0\}|^{\frac{2}{q}} \|\varrho\|^2_{L^p(\Pi)},\]
as required.

 \end{proof}
\begin{remark}
For any $\zeta\in\Xi_\varepsilon$, it is easy to check  by integration by parts that 
\[\frac{1}{2}\int_{V_\zeta}|\nabla \mathcal G\zeta|^2 dx=\frac{1}{2}\int_\Pi\zeta(\mathcal G\zeta-qx_2-\mu_\zeta).\]
Therefore Lemma \ref{lem302} in fact says that the kinetic energy of the fluid on the vortex core $V_\zeta$ has a uniform upper bound.
\end{remark}

With Lemmas \ref{lem301} and \ref{lem302} at hand, we can easily deduce a lower bound for the Lagrangian multiplier $\mu_\zeta.$
\begin{lemma}\label{lem303}
There exists some $C_3>0,$ depending only on $\varrho$ and $q$, such that
\begin{equation}\label{zyw1}
\mu_\zeta\geq -\frac{\kappa}{2\pi}\ln\varepsilon-C_3,\quad\forall \,\zeta\in\Xi_\varepsilon.
\end{equation}
As a consequence, for any $\zeta\in\Xi_\varepsilon$ and $\mathbf x\in V_\zeta$ it holds that
\begin{equation}\label{zyw2}
\mathcal G\zeta(\mathbf x)\geq -\frac{\kappa}{2\pi}\ln\varepsilon+qx_2-C_3.
\end{equation}
\end{lemma}
\begin{proof}
Fix $\zeta\in\Xi_\varepsilon$. In view of the following equality
\[(E-qI)(\zeta)=\frac{1}{2}\int_\Pi\zeta(\mathcal G\zeta-qx_2-\mu_\zeta)d\mathbf x-\frac{q}{2}\int_\Pi x_2\zeta d\mathbf x+\frac{\kappa}{2}\mu_\zeta,\]
we have
\begin{align*}
\mu_\zeta&=\frac{2}{\kappa}(E-qI)(\zeta)-\frac{1}{\kappa}\int_\Pi\zeta(\mathcal G\zeta-qx_2-\mu_\zeta)d\mathbf x+\frac{q}{\kappa}\int_\Pi x_2\zeta d\mathbf x\\
&\geq -\frac{\kappa}{2\pi}\ln\varepsilon-\frac{2C_1+C_2}{\kappa}.
\end{align*}
Here we used Lemma \ref{lem301} and Lemma \ref{lem302}. Hence \eqref{zyw1} is proved. Combining \eqref{lesss} and \eqref{zyw1}, we get \eqref{zyw2}.

\end{proof}

The following lemma is an improvement of Lemma 1 in \cite{B11}.

\begin{lemma}\label{lem304}
There exists some $C_4>0,$ depending only on $\varrho$ and $q$, such that for any $\zeta\in \Xi_\varepsilon$, it holds that
\[\mathcal G\zeta(\mathbf x)\leq -\frac{\kappa}{2\pi}\ln\varepsilon+C_4(1+\ln x_2),\quad\forall \,\mathbf x=(x_1,x_2)\in\Pi,\,x_2\geq1.\]
\end{lemma}
\begin{proof}
For $\zeta\in\Xi_\varepsilon,$ we have
\begin{align*}
\mathcal G\zeta(\mathbf x)&=\frac{1}{2\pi}\int_\Pi \ln\frac{|\mathbf x-\bar{\mathbf y}|}{|\mathbf x-\mathbf y|}\zeta(\mathbf y)d\mathbf y\\
&=\frac{1}{2\pi}\int_{\mathbf y\in\Pi, |\mathbf y-\mathbf x|\geq x_2}\ln\frac{|\mathbf x-\bar{\mathbf y}|}{|\mathbf x-\mathbf y|}\zeta(\mathbf y)d\mathbf y+\frac{1}{2\pi}\int_{\mathbf y\in\Pi, |\mathbf y-\mathbf x|\leq x_2}\ln\frac{|\mathbf x-\bar{\mathbf y}|}{|\mathbf x-\mathbf y|}\zeta(\mathbf y)d\mathbf y\\
&:=A+B.
\end{align*}
Notice that when $|\mathbf x-\mathbf y|\geq x_2$ we have
\[|\mathbf x-\bar{\mathbf y}|\leq |\bar{\mathbf x}-{\mathbf x}|+|\mathbf x-\mathbf y|= 2x_2+|\mathbf x-\mathbf y|\leq 3|\mathbf x-\mathbf y|,\]
hence
\[A=\frac{1}{2\pi}\int_{\mathbf y\in\Pi, |\mathbf y-\mathbf x|\geq x_2}\ln\frac{|\mathbf x-\bar{\mathbf y}|}{|\mathbf x-\mathbf y|}\zeta(\mathbf y)d\mathbf y\leq \frac{\ln3}{2\pi}\int_\Pi\zeta(\mathbf y)d\mathbf y=\frac{\ln3}{2\pi}\kappa.\]
To estimate  $B$, notice that when  $|\mathbf x-\mathbf y|\leq x_2$ it holds that
\[|\mathbf x-\bar{\mathbf y}|\leq |\bar{\mathbf x}-{\mathbf x}|+|\mathbf x-\mathbf y|\leq 2x_2+|\mathbf x-\mathbf y|\leq 3x_2,\]
therefore
\begin{equation}\label{2003}
\begin{split}
B&=\frac{1}{2\pi}\int_{\mathbf y\in\Pi, |\mathbf y-\mathbf x|\leq x_2}\ln\frac{|\mathbf x-\bar{\mathbf y}|}{|\mathbf x-\mathbf y|}\zeta(\mathbf y)d\mathbf y\\
&\leq \frac{1}{2\pi}\int_\Pi\ln(3x_2)\zeta(\mathbf y)d\mathbf y+\frac{1}{2\pi}\int_\Pi\ln\frac{1}{|\mathbf x-\mathbf y|}\zeta(\mathbf y)d\mathbf y\\
&\leq \frac{\ln(3x_2)}{2\pi}\kappa+\frac{1}{2\pi}\int_{B_{\varepsilon a}(\mathbf 0)}\ln\frac{1}{|\mathbf y|}\zeta^*(\mathbf y)d\mathbf y\\
&= \frac{\ln(3x_2)}{2\pi}\kappa-\frac{\kappa}{2\pi}\ln\varepsilon+\frac{\kappa}{2\pi}\int_{B_{a}(\mathbf 0)}\ln\frac{1}{|\mathbf y|}\varrho^*(\mathbf y)d\mathbf y.
\end{split}
\end{equation}
Here we used the rearrangement inequality in Lemma \ref{rri1}.
\end{proof}

Now we can give a uniform bound for the vortex core $V_\zeta$ in the $x_2$ direction.

\begin{lemma}\label{lem305}
There exists some $C_5>0,$ depending only on $\varrho$ and $q$, such that
 \begin{equation}\label{ccp1}
V_\zeta\subset\mathbb R\times[0, C_5],\quad\forall \,\zeta\in\Xi_\varepsilon.
 \end{equation}
 As a consequence,  it holds that
 \begin{equation}\label{imup}
 I(\zeta)\leq C_5\kappa,\quad \forall \,\zeta\in\Xi_\varepsilon.
 \end{equation}
\end{lemma}

\begin{proof}
Fix $\zeta\in\Sigma_\varepsilon$ and $\mathbf x=(x_1,x_2)\in V_\zeta$.
Without loss of generality we assume $x_2\geq 1$, then Lemma \ref{lem304} can be applied, which gives
 \begin{equation}\label{ape1}
 \mathcal G\zeta(\mathbf x)\leq -\frac{\kappa}{2\pi}\ln\varepsilon+C_4(1+\ln x_2).
 \end{equation}
On the other hand, recalling \eqref{zyw2} we have
 \begin{equation}\label{ape2}
 \mathcal G\zeta(\mathbf x)\geq-\frac{\kappa}{2\pi}\ln\varepsilon+qx_2-C_3.
 \end{equation}
From \eqref{ape1} and \eqref{ape2} we obtain
\[qx_2-C_3\leq C_4(1+\ln x_2),\]
which implies the existence of some $C_5>0,$ depending only on $C_3$ and $C_4$, such that
$x_2\leq C_5.$
Hence the proof is finished.
\end{proof}

To deduce a bound for ${\rm supp}(\zeta)$ in the $x_1$ direction, we need the following lemma. Recall the definition of  $\Xi^0_\varepsilon$ in \eqref{xi0}.
\begin{lemma}\label{lem3060}
There exists some $C_6>0,$ depending only on $\varrho$ and $q$, such that
for any $\zeta\in\Xi^0_\varepsilon$  it holds that
\[\mathcal G\zeta(\mathbf x)\leq C_6\varepsilon^{2/p-2}|x_1|^{-1/(2p)},\quad\forall\,\mathbf x=(x_1,x_2)\in \Pi, \,|x_1|\geq 1.\]

\end{lemma}
\begin{proof}
Fix $\zeta\in\Xi^0_\varepsilon$.  Then $\zeta$ is Steiner-symmetric in the $x_2$ axis.
Choosing $s=p$ in Lemma \ref{lemm4} and applying it to $\zeta$, we get for any
$\mathbf x\in\Pi,$ $|x_1|\geq 1$ that
\begin{equation}
\mathcal G\zeta(\mathbf x)\leq K\left(I(\zeta)+\|\zeta\|_{L^1(\Pi)}+\|\zeta\|_{L^p(\Pi)}\right)x_2|x_1|^{-1/(2p)}.
\end{equation}
Taking into account \eqref{ccp1}, \eqref{imup}, and the following fact
\[\|\zeta\|_{L^p(\Pi)}=\|\varrho^\varepsilon\|_{L^p(\Pi)}=\varepsilon^{2/p-2}\|\varrho\|_{L^p(\Pi)},\]
we further get 
\begin{align*}
\mathcal G\zeta(\mathbf x)&\leq K\left(C_5\kappa+\kappa+\varepsilon^{2/p-2}\|\varrho\|_{L^p(\Pi)}\right)C_5|x_1|^{-1/(2p)}\\
&= KC_5\varepsilon^{2/p-2}\left((C_5 \kappa+\kappa)\varepsilon^{2-2/p}+\|\varrho\|_{L^p(\Pi)}\right)|x_1|^{-1/(2p)}\\
&\leq KC_5\varepsilon^{2/p-2}\left((C_5 \kappa+\kappa)\varepsilon_0^{2-2/p}+\|\varrho\|_{L^p(\Pi)}\right)|x_1|^{-1/(2p)}.
\end{align*}
The desired estimate follows by choosing
\[C_6=KC_5 \left((C_5 \kappa+\kappa)\varepsilon_0^{2-2/p}+\|\varrho\|_{L^p(\Pi)}\right).\]

\end{proof}

Using Lemma \ref{lem3060},
we can  give a  rough  bound for the vortex core $V_\zeta$ in the $x_1$ direction as follows.
\begin{lemma}\label{lem306}
There exists some $C_7>0,$ depending only on $\varrho$ and $q$, such that
 \begin{equation}\label{ccp2}
{\rm diam}(V_\zeta)\leq C_7\varepsilon^{4-4p},\quad\forall\,\zeta\in\Xi_\varepsilon. \end{equation}
\end{lemma}
\begin{proof}
Without loss of generality,  we  assume, up to a translation in the $x_1$ direction, that  $\zeta\in\Xi^0_\varepsilon.$ For an arbitrary point $\mathbf x\in V_\zeta$,  by \eqref{zyw2} we have
\[\mathcal G\zeta(\mathbf x)\geq -\frac{\kappa}{2\pi}\ln\varepsilon+qx_2-C_3\geq -\frac{\kappa}{2\pi}\ln \varepsilon-C_3,\]
which together with Lemma \ref{lem3060} yields
\begin{equation}\label{fore}
-\frac{\kappa}{2\pi}\ln\varepsilon-C_3\leq C_6\varepsilon^{2/p-2}|x_1|^{-1/(2p)}.
\end{equation}
Without loss of generality,   assume that   $\varepsilon_0$ is sufficiently small such that
\[-\frac{\kappa}{2\pi}\ln \varepsilon_0-C_3\geq 1.\]
Then
\begin{equation}\label{fort}
C_6\varepsilon^{2/p-2}|x_1|^{-1/(2p)}\geq 1,
\end{equation}
which yields
\begin{equation}\label{entro}
|x_1|\leq C^{2p}_6\varepsilon^{4-4p}.
\end{equation}
Since $\mathbf x\in {\rm supp}(\zeta)$ is arbitrarily chosen, we obtain
\begin{equation}\label{ccp8}
 {\rm supp}(\zeta)\subset [-C^{2p}_6\varepsilon^{4-4p},C^{2p}_6\varepsilon^{4-4p}]\times\mathbb R,\quad\forall \,\zeta\in\Xi_\varepsilon.
 \end{equation}
Then \eqref{ccp2}  follows from \eqref{ccp8} and Lemma \ref{lem305}.\end{proof}

Having made enough preparations, we are ready to deduce a uniform bound for the size of  the vortex core $V_\zeta$

\begin{lemma}\label{lem307}
There exists some $C_8>0,$ depending only on $\varrho$ and $q$, such that
 \begin{equation}\label{ccp5}
 {\rm diam}(V_\zeta)\leq C_8,\quad\forall\,\zeta\in\Xi_\varepsilon. \end{equation}
\end{lemma}
\begin{proof}

We prove \eqref{ccp5} by contradiction. Assume there exist  $\{\varepsilon_n\}_{n=1}^{+\infty}\subset(0,\varepsilon_0)$ and  $\{\zeta_n\}_{n=1}^{+\infty}\subset\Xi_{{\varepsilon_n}}$ such that
\[d_n:={\rm diam}(V_{\zeta_n})\to+\infty.\]
By Lemma \ref{lem306}, we have 
\begin{equation}\label{rp1i}
d_n\leq C_7\varepsilon_n^{4-4p}.
\end{equation}
Hence it necessarily holds that
\begin{equation}\label{rp2i}
\lim_{n\to+\infty}\varepsilon_n=0.
\end{equation}
Without loss of generality, we assume that   $\zeta_n\in\Xi^0_{\varepsilon_n}$ for all $n$. 
Below we deduce a contradiction.

For any $\mathbf x\in V_{\zeta_n}$, by \eqref{zyw2} it holds that
\[\mathcal G\zeta_n(\mathbf x)\geq -\frac{\kappa}{2\pi}\ln\varepsilon_n-C_3,\]
which together with the fact that $\|\zeta_n\|_{L^1(\Pi)}=\kappa$ gives
\[\int_\Pi\ln\frac{\varepsilon_n}{|\mathbf x-\mathbf y|}\zeta_n(\mathbf y)d\mathbf y\geq -\int_{V_{\zeta_n}}\ln|\mathbf x-\bar{\mathbf y}|\zeta_n(\mathbf y)d\mathbf y-2\pi C_3.\]
Observe that for any $\mathbf x,\mathbf y\in V_{\zeta_n}$  
\[|\mathbf x-\bar{\mathbf y}|\leq |\mathbf x-\mathbf y|+|\mathbf y-\bar{\mathbf y}|\leq d_n+2C_5.
\]
Hence  we get
\begin{equation}\label{rp1}
\int_\Pi\ln\frac{\varepsilon_n}{|\mathbf x-\mathbf y|}\zeta_n(\mathbf y)d\mathbf y\geq -\kappa\ln(d_n+2C_5)-2\pi C_3.
\end{equation}
Divide the  integral in \eqref{rp1} into two parts
\begin{align}\label{youu}
\int_\Pi\ln\frac{\varepsilon_n}{|\mathbf x-\mathbf y|}\zeta_n(\mathbf y)d\mathbf y&= \int_{|\mathbf x-\mathbf y|\leq \frac{d_n}{N}}\ln\frac{\varepsilon_n}{|\mathbf x-\mathbf y|}\zeta_n(\mathbf y)d\mathbf y+
\int_{|\mathbf x-\mathbf y|\geq \frac{d_n}{N}}\ln\frac{\varepsilon_n}{|\mathbf x-\mathbf y|}\zeta_n(\mathbf y)d\mathbf y,
\end{align}
where $N$ is a large positive integer such that
\begin{equation}\label{larn}
N>40p-30.
\end{equation}
Note that $N>50$ since we have assumed $p>2.$
For the first part, it is easy to check by using the rearrangement inequality in Lemma \ref{rri1} (as in \eqref{2003}) that
\begin{equation}\label{rp2}
\int_{|\mathbf x-\mathbf y|\leq \frac{d_n}{N}}\ln\frac{\varepsilon_n}{|\mathbf x-\mathbf y|}\zeta_n(\mathbf y)d\mathbf y\leq \int_{\Pi}\ln\frac{\varepsilon_n}{|\mathbf x-\mathbf y|}\zeta_n(\mathbf y)d\mathbf y\leq \int_{\mathbb R^2}\ln\frac{1}{|\mathbf y|}\varrho^*(\mathbf y)d\mathbf y\leq K_1
\end{equation}
for some  positive constant $K_1$ depending only on $\varrho$.
For the second part, we have
\begin{equation}\label{rp3}
\int_{|\mathbf x-\mathbf y|\geq \frac{d_n}{N}}\ln\frac{\varepsilon_n}{|\mathbf x-\mathbf y|}\zeta_n(\mathbf y)d\mathbf y\leq \ln\frac{N\varepsilon_n}{d_n}\int_{|\mathbf x-\mathbf y|\geq \frac{d_n}{N}}\zeta_n(\mathbf y)d\mathbf y.
\end{equation}
Combining \eqref{rp1}, \eqref{rp2} and \eqref{rp3} we get
\begin{equation}\label{nbs3}
\ln\frac{N\varepsilon_n}{d_n}\int_{|\mathbf x-\mathbf y|\geq \frac{d_n}{N}}\zeta_n(\mathbf y)d\mathbf y\geq -\kappa\ln(d_n+2C_5)-2\pi C_3-K_1.
\end{equation}
Since we have assumed $d_n\to+\infty$ as $n\to+\infty,$ below we can assume that  $n$ is sufficiently large such that 
\[N\varepsilon_n<d_n,\quad d_n\geq 2C_5.\]
Then from \eqref{nbs3} we have 
\begin{equation}\label{dpt1}
\begin{split}
\int_{|\mathbf x-\mathbf y|\geq \frac{d_n}{N}}\zeta_n(\mathbf y)d\mathbf y&\leq \frac{\kappa\ln(d_n+2C_5)+2\pi C_3+K_1}{\ln d_n-\ln\varepsilon_n-\ln N}\\
&\leq \frac{\kappa\ln d_n+\kappa\ln2+2\pi C_3+K_1}{\ln d_n-\ln\varepsilon_n-\ln N}\\
&\leq  \frac{\kappa\ln d_n+K_2}{\ln d_n-\ln\varepsilon_n-K_2},
\end{split}
\end{equation}
where $K_2=\kappa\ln2+2\pi C_3+K_1+\ln N.$
To proceed, notice that for  sufficiently large $n$ such that $\varepsilon_n<e^{-K_2(1+\kappa^{-1})},$
the function
\[f(s)=\frac{\kappa s+K_2}{s-\ln\varepsilon_n-K_2}\]
is increasing for $s\in(0,+\infty).$ Taking into account \eqref{rp1i}, we get
\begin{equation}\label{dpt2}
\frac{\ln d_n+K_2}{\ln d_n-\ln\varepsilon_n-K_2}\leq \frac{\kappa\ln \left(C_7\varepsilon_n^{4-4p}\right)+K_2}{\ln \left(C_7\varepsilon_n^{4-4p}\right)-\ln\varepsilon_n-K_2}\to \frac{4p-4}{4p-3}\kappa
\end{equation}
as $n\to+\infty.$  Note that we used   \eqref{rp2i} when passing to the limit  $n\to+\infty$ in \eqref{dpt2}. Therefore we get from \eqref{dpt1} and \eqref{dpt2} that for sufficiently large $n$
\[\int_{|\mathbf x-\mathbf y|\geq \frac{d_n}{N}}\zeta_n(\mathbf y)d\mathbf y<  \frac{4p-4}{4p-3}\kappa,\]
or equivalently,
\begin{equation}\label{moned}
\int_{|\mathbf x-\mathbf y|\leq \frac{d_n}{N}}\zeta_n(\mathbf y)d\mathbf y> \frac{1}{4p-3}\kappa.
\end{equation}
Recalling \eqref{larn}, we get 
\begin{equation}\label{mone}
\int_{|\mathbf x-\mathbf y|\leq \frac{d_n}{N}}\zeta_n(\mathbf y)d\mathbf y> \frac{10}{N}\kappa,
\end{equation}
provided that $n$ is large enough.

Based on \eqref{mone}, we can deduce a contradiction. To this end, 
denote
\[e_n:=\max_{\mathbf x\in {V_{\zeta_n}}}x_1.\]
Then it is easy to see that
\[ d_n\leq 2e_n+C_5.\]
Hence
\begin{equation*} 
\int_{|\mathbf x-\mathbf y|\leq \frac{d_n}{N}}\zeta_n(\mathbf y)d\mathbf y\leq \int_{|x_1-y_1|\leq \frac{d_n}{N}}\zeta_n(\mathbf y)d\mathbf y\leq \int_{|x_1-y_1|\leq \frac{2e_n+C_5}{N}}\zeta_n(\mathbf y)d\mathbf y\leq \int_{|x_1-y_1|\leq \frac{4e_n}{N}}\zeta_n(\mathbf y)d\mathbf y,
\end{equation*}
provided that $n$ is large enough (such that $e_n\geq C_5/2$).
This together with \eqref{mone} yields for sufficiently large $n$
\begin{equation}\label{mmki}
\int_{|x_1-y_1|\leq \frac{4e_n}{N}}\zeta_n(\mathbf y)d\mathbf y>\frac{10}{N}\kappa,\quad\forall\, \mathbf x\in V_{\zeta_n}.
\end{equation}
Below fix a large $n$ such that \eqref{mmki} holds.
Since $\zeta_n$ is Steiner-symmetric in the $x_2$ axis (note that we have assumed    $\zeta_n\in\Xi^0_{\varepsilon_n}$), we have 
\begin{equation}\label{con11}
\int_{|s-y_1|\leq \frac{4e_n}{N}}\zeta_n(\mathbf y)d\mathbf y>\frac{10}{N}\kappa,\quad\forall\, s\in(-e_n,e_n).
\end{equation}
Define 
\[s_k=\frac{4(2k-1)e_n}{N},\quad k=1,\cdot\cdot\cdot,{\left\lceil\frac{N}{10}\right\rceil}.\]
Here  $\lceil a\rceil$ is the smallest integer not less than $a$ for any $a\in\mathbb R$. Then it is easy to check that for any $1\leq k\leq {\left\lceil\frac{N}{10}\right\rceil},$
\[ 0< s_k\leq \frac{4\left(2{\left\lceil\frac{N}{10}\right\rceil}-1\right)e_n}{N}\leq \frac{4\left(2\left(\frac{N}{10}+1\right)-1\right)e_n}{N}=\left(\frac{4}{5}+\frac{4}{N}\right)e_n<e_n.\]
Here we used the fact that $N>50.$
Therefore by \eqref{con11} we get
\begin{equation}\label{con11}
\int_{|s_k-y_1|\leq \frac{4e_n}{N}}\zeta_n(\mathbf y)d\mathbf y>\frac{10}{N}\kappa,\quad k=1,\cdot\cdot\cdot,{\left\lceil\frac{N}{10}\right\rceil}.
\end{equation}
Summing over $k=1,\cdot\cdot\cdot,{\left\lceil\frac{N}{10}\right\rceil},$ we obtain
\begin{equation}\label{coniii}
\int_\Pi\zeta_n(\mathbf y)d\mathbf y\geq \sum_{k=1}^{{\left\lceil\frac{N}{10}\right\rceil}}\int_{|y_1-s_k|< \frac{4e_n}{N}}\zeta_n(\mathbf y)d\mathbf y> {\left\lceil\frac{N}{10}\right\rceil}\frac{10}{N}\kappa\geq\kappa.
\end{equation}
Note that in the first inequality in \eqref{coniii} we used the following fact  
\[\left\{\mathbf y\in\Pi\mid |y_1-s_j|< \frac{4e_n}{N}\right\}\cap \left\{\mathbf y\in\Pi\mid |y_1-s_k|< \frac{4e_n}{N}\right\}=\varnothing, \quad \forall \,1\leq j\neq k\leq  {\left\lceil\frac{N}{10}\right\rceil}.\]
Obviously \eqref{coniii}  contradicts the fact that $\|\zeta_n\|_{L^1(\Pi)}=\kappa.$ The proof is finished.

\end{proof}

With Lemma \ref{lem307} at hand, we can easily deduce a better bound for the   size of the vortex core $V_{\zeta}$.
\begin{lemma}\label{lem3010}
There exists some $C_9>0,$ depending only on $\varrho$ and $q$, such that
 \begin{equation}\label{ccp10}
 {\rm diam}(V_\zeta)\leq C_9\varepsilon,\quad\forall\,\zeta\in\Xi_\varepsilon. \end{equation}
\end{lemma}
\begin{proof}
Fix $\zeta\in\Xi_\varepsilon$. For any $\mathbf x\in V_\zeta,$  by \eqref{zyw2}  we have
\[\mathcal G\zeta (\mathbf x)\geq -\frac{\kappa}{2\pi}\ln\varepsilon -C_3,\]
or equivalently
\begin{equation}\label{twis}
\int_\Pi\ln\frac{\varepsilon }{|\mathbf x-\mathbf y|}\zeta(\mathbf y)d\mathbf y\geq -\int_{V_\zeta}\ln|\mathbf x-\bar{\mathbf y}|\zeta(\mathbf y)d\mathbf y-2\pi C_3.
\end{equation}
Notice by Lemma \ref{lem305} and Lemma \ref{lem307} that
\begin{equation}\label{twis2}
|\mathbf x-\bar{\mathbf y}|\leq |\mathbf x-\mathbf y|+|\mathbf y-\bar{\mathbf y}|\leq C_8+2C_5,\quad\forall\,\mathbf x,\mathbf y\in V_\zeta.
\end{equation}
Therefore we get
\begin{equation}\label{twis3}
\int_\Pi\ln\frac{\varepsilon }{|\mathbf x-\mathbf y|}\zeta(\mathbf y)d\mathbf y\geq -\kappa\ln (C_8+2C_5)-2\pi C_3.
\end{equation}
To proceed, we divide the integral in \eqref{twis3} into two parts
\begin{align*}
\int_\Pi\ln\frac{\varepsilon }{|\mathbf x-\mathbf y|}\zeta(\mathbf y)d\mathbf y&=\int_{|\mathbf x-\mathbf y|\leq R\varepsilon}\ln\frac{\varepsilon }{|\mathbf x-\mathbf y|}\zeta(\mathbf y)d\mathbf y+\int_{|\mathbf x-\mathbf y|\geq R\varepsilon}\ln\frac{\varepsilon }{|\mathbf x-\mathbf y|}\zeta(\mathbf y)d\mathbf y,
\end{align*}
where $R>1$ is a positive number to be determined later.
As in \eqref{rp2}, we can use the rearrangement inequality in Lemma \ref{rri1} to estimate the first integral  to get
\[\int_{|\mathbf x-\mathbf y|\leq R\varepsilon}\ln\frac{\varepsilon }{|\mathbf x-\mathbf y|}\zeta(\mathbf y)d\mathbf y\leq K_2\]
for some positive constant $K_2$ depending only on $\varrho$. For the second integral, we have
\begin{align*}
\int_{|\mathbf x-\mathbf y|\geq R\varepsilon}\ln\frac{\varepsilon }{|\mathbf x-\mathbf y|}\zeta(\mathbf y)d\mathbf y&\leq -\ln R\int_{|\mathbf x-\mathbf y|\geq R\varepsilon}\zeta(\mathbf y)d\mathbf y.
\end{align*}
Hence we obtain
\[\int_{|\mathbf x-\mathbf y|\geq R\varepsilon}\zeta(\mathbf y)d\mathbf y\leq \frac{\kappa\ln(C_8+2C_5)+2\pi C_3+K_2}{\ln R}.\]
By choosing $R$ sufficiently large, depending only on $\varrho$ and $q$, such that
\[\frac{\kappa\ln(C_8+2C_5)+2\pi C_3+K_2}{\ln R}\leq \frac{\kappa}{3},\]
 we get
\[\int_{|\mathbf x-\mathbf y|\geq R\varepsilon}\zeta(\mathbf y)d\mathbf y\leq \frac{\kappa}{3}.\]
Therefore
\begin{equation}\label{havv}
\int_{|\mathbf x-\mathbf y|\leq R\varepsilon}\zeta(\mathbf y)d\mathbf y\geq \frac{2}{3}\kappa.
\end{equation}
Since \eqref{havv} holds true for any $\mathbf x\in V_\zeta$, we have
\begin{equation}\label{havv}
{\rm diam}(V_\zeta)\leq 2R\varepsilon.
\end{equation}
In fact, suppose otherwise ${\rm diam}(V_\zeta)>2 R\varepsilon.$ Then there exist two points $\mathbf x_1,\mathbf x_2\in V_\zeta$ such that $|\mathbf x_1-\mathbf x_2|>2R\varepsilon.$ It is obvious that  $B_{R\varepsilon}(\mathbf x_1)\cap B_{R\varepsilon}(\mathbf x_2)=\varnothing$, therefore
\[\int_\Pi\zeta(\mathbf y)d\mathbf y\geq \int_{|\mathbf x_1-\mathbf y|\leq R\varepsilon}\zeta(\mathbf y)d\mathbf y+\int_{|\mathbf x_2-\mathbf y|\leq R\varepsilon}\zeta(\mathbf y)d\mathbf y\geq \frac{4}{3}\kappa,\]
  contradicting  the fact that $\|\zeta\|_{L^1(\Pi)}=\kappa.$
The desired result follows by choosing $C_9=2R.$
\end{proof}

As in Theorem \ref{thm1},  for any $\zeta\in\Xi^0_\varepsilon$  we define
\[\mathbf x^{\zeta,\varepsilon}= \frac{1}{\kappa}\int_\Pi \mathbf x\zeta(\mathbf x)d\mathbf x.\]
By symmetry, it is clear that  $ x_1^{\zeta,\varepsilon}\equiv 0$.

\begin{lemma}\label{yayaya}
For any $\zeta\in\Xi^0_\varepsilon,$ it holds that
\begin{equation}\label{yayaya1}
V_\zeta\subset   \overline{B_{C_9\varepsilon}(\mathbf x^{\zeta,\varepsilon}) },
\end{equation}
where $C_9$ is determined by \eqref{ccp10}. 
\end{lemma}
\begin{proof}

For any $\mathbf z\in V_\zeta,$ by Lemma \ref{lem3010} we have
\[|\mathbf z-\mathbf x^{\zeta,\varepsilon}|=\left|\frac{1}{\kappa}\int_{V_\zeta}(\mathbf z-\mathbf x)\zeta(\mathbf x)d\mathbf x\right|\leq C_9\varepsilon,\]
which yields
\[\{\mathbf x\in\Pi\mid \zeta(\mathbf x)>0\}\subset \overline{B_{C_9\varepsilon}(\mathbf x^{\zeta,\varepsilon})}. \]
Hence \eqref{yayaya1} follows immediately.
\end{proof}
The following lemma is about the limiting position of $\mathbf x^{\zeta,\varepsilon}.$
\begin{lemma}\label{lem3020}
For any $\epsilon>0,$ there exists some  $\varepsilon_1\in(0,\varepsilon_0),$ such that for any $\varepsilon\in(0,\varepsilon_1),$ it holds that
\[|\mathbf x^{\zeta,\varepsilon}-\hat{\mathbf x}|<\epsilon,\quad\forall\,\zeta\in\Xi^0_\varepsilon,\]
where $\hat{\mathbf x}=(0,\kappa/(4\pi q))$ as in Theorem \ref{thm1}.

\end{lemma}
\begin{proof}
Fix a sequence  $\{\varepsilon_n\}_{n=1}^{+\infty}\subset(0,\varepsilon_0)$ such that   $\varepsilon_n\to0^+$ as $n\to+\infty$, and a sequence  $\{\zeta_n\}_{n=1}^{+\infty}, $ $\zeta_n\in  \Xi_{\varepsilon_n}^0$. Denote 
\[\mathbf x^n=\frac{1}{\kappa}\int_\Pi \mathbf x\zeta^n(\mathbf x)d\mathbf x.\]
It suffices to show that $\mathbf x^n\to\hat{\mathbf x}$ as $n\to+\infty.$

First by symmetry we have  $x^n_1\equiv 0, \forall\,n.$ Moreover, by Lemma \ref{lem305} we have  $0\leq x^n_2\leq C_5.$ 
Without loss of generality, we assume that 
\begin{equation}\label{789}
C_5>\frac{\kappa}{4\pi q}.
\end{equation}
Up to a subsequence, we assume that $\mathbf x^n\to \tilde{\mathbf x}$ as $n\to+\infty.$ Hence $\tilde x_1=0,0\leq \tilde x_2\leq C_5.$

Define 
\begin{equation}\label{defov}
v_n(\mathbf x)= \frac{1}{\varepsilon_n^2}\rho^*\left(\frac{\mathbf x-\hat{ \mathbf x}}{\varepsilon_n}\right),
\end{equation}
then $v_n\in \mathcal R(\varrho^{\varepsilon_n})$ if $n$ is sufficiently large, which implies
\[(E-qI)(\zeta_n)\geq (E-qI)(v_n).\]
For $(E-qI)(\zeta_n),$ we have
\begin{equation}\label{misu}
(E-qI)(\zeta_n)=\frac{1}{4\pi}\int_\Pi\int_\Pi\ln\frac{|\mathbf x-\bar{\mathbf y}|}{|\mathbf x-\mathbf y|}\zeta_n(\mathbf x)\zeta_n(\mathbf y)d\mathbf xd\mathbf y-q\int_\Pi x_2\zeta_n(\mathbf x)d\mathbf x.
\end{equation}
For $(E-qI)(v_n),$ we have
\begin{equation}\label{misu}
(E-qI)(v_n)=\frac{1}{4\pi}\int_\Pi\int_\Pi\ln\frac{|\mathbf x-\bar{\mathbf y}|}{|\mathbf x-\mathbf y|} v_n(\mathbf x)v_n(\mathbf y)d\mathbf xd\mathbf y-q\int_\Pi x_2v_n(\mathbf x)d\mathbf x.
\end{equation}
By the rearrangement inequality in Lemma \ref{rri2} we have
\[\frac{1}{4\pi}\int_\Pi\int_\Pi\ln\frac{1}{|\mathbf x-\mathbf y|}\zeta_n(\mathbf x)\zeta_n(\mathbf y)d\mathbf xd\mathbf y\leq \frac{1}{4\pi}\int_\Pi\int_\Pi\ln\frac{1}{|\mathbf x-\mathbf y|}v_n(\mathbf x)v_n(\mathbf y)d\mathbf xd\mathbf y,\]
hence we get
\begin{align*}&\frac{1}{4\pi}\int_\Pi\int_\Pi\ln{|\mathbf x-\bar{\mathbf y}|}\zeta_n(\mathbf x)\zeta_n(\mathbf y)d\mathbf xd\mathbf y-q\int_\Pi x_2\zeta_n(\mathbf x)d\mathbf x\\
\geq &\frac{1}{4\pi}\int_\Pi\int_\Pi\ln{|\mathbf x-\bar{\mathbf y}|}v_n(\mathbf x)v_n(\mathbf y)d\mathbf xd\mathbf y-q\int_\Pi x_2v_n(\mathbf x)d\mathbf x.
\end{align*}
Using Lemma \ref{yayaya} and passing to  the limit $n\to+\infty$, 
 we get
\[\frac{\kappa^2}{4\pi}\ln(2\tilde x_2)-\kappa q\tilde x_2\geq \frac{\kappa^2}{4\pi}\ln(2  \hat x_2)-\kappa q \hat x_2.\]
Here we used the convention $\frac{\kappa^2}{4\pi}\ln(2\tilde x_2)-\kappa q\tilde x_2=-\infty $ if $\tilde x_2=0.$
This means that $\tilde x_2$ is a maximum point of the function
\[h(s)=\frac{\kappa^2}{4\pi}\ln(2s)-\kappa qs\]
on $[0,C_5]$. On the other hand, it is easy to check that since $\hat x_2=\kappa/(4\pi q)$ is the unique maximum point of $h$ on $[0,+\infty).$ Taking into account \eqref{789},  it must hold $\tilde x_2=\hat x_2$.  This finishes the proof.
\end{proof}

As in Theorem \ref{thm1},  for any $\zeta\in\Xi^0_\varepsilon$ we extend it to $\mathbb R^2$ by setting $\zeta\equiv0$ in the lower half-plane. Denote 
\begin{equation}\label{ggla}
\nu^{\zeta,\varepsilon}(\mathbf x)=\varepsilon^2\zeta(\varepsilon\mathbf x+\mathbf x^{\zeta,\varepsilon}).
\end{equation}

\begin{lemma}\label{lem3030}
For any  $\epsilon>0,$ there exists some  $\varepsilon_2\in(0,\varepsilon_0),$ such that for any $\varepsilon\in(0,\varepsilon_2),$ it holds that
\[\|\nu^{\zeta,\varepsilon}-\varrho^*\|_{L^p(\mathbb R^2)}<\epsilon,\quad\forall\,\zeta\in\Xi^0_\varepsilon.\]

\end{lemma}
\begin{proof}

Fix a sequence  $\{\varepsilon_n\}_{n=1}^{+\infty}\subset(0,\varepsilon_0)$ such that   $\varepsilon_n\to0^+$ as $n\to+\infty$, and a sequence  $\{\zeta_n\}_{n=1}^{+\infty}, \zeta_n\in  \Xi_{\varepsilon_n}^0$. 
It suffices to show that $\nu^{\zeta_n,\varepsilon_n}\to\varrho^*$ in $L^p(\mathbb R^2)$ as $n\to+\infty.$ For simplicity, we denote $\nu_n=\nu^{\zeta_n,\varepsilon_n}.$ By  the definition of $\nu^{\zeta_n,\varepsilon_n}$ (see \eqref{ggla}), it is clear that 
 \begin{equation}\label{yyss}
 \zeta_n(\mathbf x)=\frac{1}{\varepsilon_n^2}\nu_n\left(\frac{\mathbf x-\mathbf x^{\zeta_n,\varepsilon_n}}{\varepsilon_n}\right).
 \end{equation}
Moreover, for each $n$ it holds that
\begin{equation}\label{eassy}
\int_{\mathbb R^2}\mathbf x\nu_n(\mathbf x)d\mathbf x=\mathbf 0, \quad\nu_n\in \mathcal R(\varrho,\mathbb R^2),\quad {\rm supp}(\nu_n)\subset \overline{B_{C_9}(\mathbf 0)}.
\end{equation}
Here $ \mathcal R(\varrho,\mathbb R^2)$ stands for the set of all equimeasurable rearrangements of $\varrho$ in $\mathbb R^2,$ that is,
\[\mathcal R(\varrho,\mathbb R^2)=\left\{v\in L_{\rm loc}^1(\mathbb R^2)\mid |\{\mathbf x\in\mathbb R^2\mid v(\mathbf x)>s\}|=|\{\mathbf x\in\Pi\mid \varrho(\mathbf x)>s\}|\,\,\forall\,s\in\mathbb R\right\}.\]
Up to a subsequence, we assume that $\nu_n\rightharpoonup \nu$ in $L^p(\mathbb R^2)$ as $n\to+\infty.$ It is easy to check that ${\rm supp}(\nu)\subset \overline{B_{C_9}(\mathbf 0)}$ and $\nu_n\rightharpoonup \nu$ in $L^2(\mathbb R^2)$ as $n\to+\infty.$

Let $v_n$ be defined by \eqref{defov}.
 Then $v_n\in\mathcal R(\varrho^{\varepsilon_n})$ if $n$ is large enough. Hence
 \[(E-qI)(\zeta_n)\geq (E-qI)(v_n),\]
which can be written as
 \begin{equation}\label{sptt5}
 \begin{split}
 & \int_\Pi\int_\Pi\ln\frac{1}{|\mathbf x-\mathbf y|}\zeta_n(\mathbf x)\zeta_n(\mathbf y)d\mathbf xd\mathbf y+ \int_\Pi\int_\Pi\ln {|\mathbf x-\bar{\mathbf y}|}\zeta_n(\mathbf x)\zeta_n(\mathbf y)d\mathbf xd\mathbf y-4\pi q\int_\Pi x_2\zeta_n(\mathbf x)d\mathbf x\\
 \geq & \int_\Pi\int_\Pi\ln\frac{1}{|\mathbf x-\mathbf y|}v_n(\mathbf x)v_n(\mathbf y)d\mathbf xd\mathbf y+ \int_\Pi\int_\Pi\ln|\mathbf x-\bar{\mathbf y}|v_n(\mathbf x)v_n(\mathbf y)d\mathbf xd\mathbf y-4\pi q\int_\Pi x_2 v_n(\mathbf x)d\mathbf x.
 \end{split}
 \end{equation}
 By the definition of $v_n$, it is clear that as $n\to+\infty$
 \begin{equation}\label{3k1}
 \frac{1}{4\pi}\int_\Pi\int_\Pi\ln|\mathbf x-\bar{\mathbf y}|v_n(\mathbf x)v_n(\mathbf y)d\mathbf xd\mathbf y-q\int_\Pi x_2 v_n(\mathbf x)d\mathbf x\to \frac{\kappa^2}{4\pi}\ln(2\hat x_2)-\kappa q \hat x_2.
 \end{equation}
By  Lemma \ref{yayaya}  and Lemma \ref{lem3020}, it is also easy to check that  as $n\to+\infty$
 \begin{equation}\label{3k2}
 \frac{1}{4\pi}\int_\Pi\int_\Pi\ln|\mathbf x-\bar{\mathbf y}|\zeta_n(\mathbf x)\zeta_n(\mathbf y)d\mathbf xd\mathbf y-q\int_\Pi x_2 \zeta_n(\mathbf x)d\mathbf x\to \frac{\kappa^2}{4\pi}\ln(2\hat x_2)-\kappa q \hat x_2.
 \end{equation}
 Therefore \eqref{sptt5}, \eqref{3k1} and \eqref{3k2} together yield
  \begin{equation}\label{sptt6}
\int_\Pi\int_\Pi\ln\frac{1}{|\mathbf x-\mathbf y|}\zeta_n(\mathbf x)\zeta_n(\mathbf y)d\mathbf xd\mathbf y\geq \int_\Pi\int_\Pi\ln\frac{1}{|\mathbf x-\mathbf y|}v_n(\mathbf x)v_n(\mathbf y)d\mathbf xd\mathbf y+h_n
 \end{equation}
 for some sequence $\{h_n\}_{n=1}^{+\infty}$ satisfying $h_n\to 0 $ as $n\to+\infty.$
 Inserting \eqref{defov} and \eqref{yyss} into \eqref{sptt6}, using the fact that $\nu_n\rightharpoonup \nu$ in $L^2(\mathbb R^2)$ as $n\to+\infty$, and passing to the limit $n\to+\infty$, we can obtain
 \[\int_{\mathbb R^2}\int_{\mathbb R^2}\ln\frac{1}{|\mathbf x-\mathbf y|}\nu(\mathbf x)\nu(\mathbf y) d\mathbf xd\mathbf y\geq \int_{\mathbb R^2}\int_{\mathbb R^2}\ln\frac{1}{|\mathbf x-\mathbf y|}\varrho^*(\mathbf x)\varrho^*(\mathbf y) d\mathbf xd\mathbf y.\]
Now we can apply Lemma \ref{bgu}  (choosing $u_n=\nu_n$ and noticing $\nu_n^*=v=\varrho^*$  therein) to get $\nu=\varrho^*.$ 
 Note that the condition \eqref{lily} in Lemma \ref{bgu} is satisfied by \eqref{eassy}.
 
To conclude, we have obtained $\nu_n\rightharpoonup \varrho^*$ in $L^p(\mathbb R^2)$ as $n\to+\infty.$ Taking into account the fact that $\nu_n,\varrho^* \in \mathcal R(\varrho,\mathbb R^2)$, we get $\|\nu_n\|_{L^p(\mathbb R^2)}=\|\varrho^*\|_{L^p(\mathbb R^2)}$ for each $n$, which implies 
$\nu_n\to \varrho^*$ in $L^p(\mathbb R^2)$ as $n\to+\infty$ by uniform convexity. 
\end{proof}

\begin{proof}[Proof of Theorem \ref{thm1}]
(i)-(iii) follow from Lemmas \ref{lemm1}-\ref{lemm3} and Lemma \ref{lem31}. (iv) -(vi) follow from Lemmas  \ref{lem3010}-\ref{lem3030}.

\end{proof}
\begin{proof}[Proof of Corollary \ref{coro1}]
Choose $q=1,\varepsilon=\lambda$ in 
 Theorem \ref{thm1} and use Lemma \ref{lem31}.

\end{proof}

\section{Further discussions}

First we discuss a slightly different  maximization problem related to this paper.
Consider the maximization of $E-qI$ over
 \begin{equation*}
 \mathcal A_\varepsilon=\{v\in\mathcal R(\varrho^\varepsilon)\mid {\rm supp}(v)\subset B_{r_0}(\hat{\mathbf x}),\,\,v\mbox{ is Steiner-symmetric about the $x_2$ axis}\},
 \end{equation*}
where $r_0>0$ is fixed. By applying Theorem 7 in \cite{B1}, it is not hard to prove that $E-qI$ attains its maximum value on $\mathcal A_\varepsilon$. Then we can perform  similar asymptotic estimates as in Section 3 to obtain a family of concentrated traveling vortex pairs with the same asymptotic behavior. Since the elements in $\mathcal A_\varepsilon$ are   supported in $B_{r_0}(\hat{\mathbf x}),$   the corresponding asymptotic estimates  become much easier. But as a cost, it is hard to analyze the stability of these vortex pairs. The biggest difficulty comes from the fact  that $\mathcal A_\varepsilon$ is not invariant for the vorticity equation.
As illustrated in \cite{B6},  in order to prove stability we need to compare the energy of the maximizers with the energy of all the elements in $\overline{\mathcal{\mathcal R(\varrho^\varepsilon)}^W}$, which is  a hard problem. In fact, it is also the main difficulty appearing in the study of  stability of other concentrated vortex flows. See \cite{CW1,CW2,W2} for example.
Now with Theorem \ref{thm1}, we immediately see that this maximization problem  in fact gives the same solutions as \eqref{mm2} does  if $\varepsilon$ is sufficiently small, from which stability follows immediately.

 Second, we discuss another variational problem related to traveling vortex pairs with prescribed impulse,  i.e., 
the maximization of   $E$ over 
\[\mathcal S_{i}=\left\{v\in\overline{\mathcal R(\varrho)^W}\mid I(v)=i\right\},\]
where $i>0$ is   fixed.
 Burton studied this problem in  \cite{B10} and  proved  that for sufficiently large $i$,  $E$ attains its maximum value on $\mathcal S_i$ and any maximizer must be contained in $\mathcal R(\varrho)$ . A form of orbital stability was also deduced. An interesting question is to study the limiting behavior for the maximizers as $i_0\to+\infty.$ Following the idea of this paper, a possible way is to study the scaled problem and  deduce suitable  asymptotic estimates. However,
after some tries, we found that the method in this paper can not be applied directly to tackle this problem. The main reason is that we are not able to obtain  appropriate uniform estimates for the traveling speed, which is not fixed  in this case.  This is an interesting further work.

{\bf Acknowledgements:}
{G. Wang was supported by National Natural Science Foundation of China (12001135, 12071098) and China Postdoctoral Science Foundation (2019M661261, 2021T140163).}

\phantom{s}
 \thispagestyle{empty}

\end{document}